\documentclass[12pt]{article}
\usepackage{geometry}
\usepackage{amsthm}
\usepackage{varioref}
\usepackage{amsmath,amstext,amsthm,a4,amssymb,amscd}   
\usepackage[T1]{fontenc}
\usepackage[utf8]{inputenc}
\usepackage{lmodern}
\usepackage{array}
\usepackage{latexsym}
\usepackage{fancyhdr}
\usepackage{mathrsfs}\let\cal\mathscr
\usepackage[all]{xy}
\usepackage{makeidx}   
\usepackage{color}
\usepackage{pifont}
\usepackage{amsfonts,amssymb}
\usepackage{hyperref}
\usepackage[numbers,square]{natbib}
\usepackage{epsf}

\usepackage{avant}
\usepackage[T1]{fontenc}      

\usepackage{amsfonts,amssymb}  
 
\usepackage{graphicx}

\newcommand \Om {\Omega}
\newcommand \om {\omega}

\newcommand{\KK}[2]{\K\big[#1,#2\big]}
\newcommand{\KKK}[3]{\K\big[#1,\K\big[#2,#3\big]\big]}

\newcommand \Spin {\Lambda( T^{*(0,1)}X)}

\DeclareMathOperator{\End}{End}

\DeclareMathOperator{\Ker}{Ker}
\DeclareMathOperator{\rk}{rk}
\DeclareMathOperator{\Cl}{Cl}
\DeclareMathOperator{\Spec}{Spec}

\newcommand \dbar {\overline{\partial}}
\newcommand \del[1] {\frac{\partial}{\partial #1}}

\newcommand \< {\mathcal{h}}
\renewcommand \> {\mathcal{i}}

\newcommand \cinf {\CC^\infty}

\renewcommand \epsilon {\varepsilon}

\newcommand \CC {{\cal C}}

\newcommand \K {{\mathcal K}}
\newcommand \LL {{\cal L}}

\newcommand \OO {\mathcal{O}}
\newcommand \PP {{\cal P}}

\newcommand \Q {{\mathcal Q}}

\newcommand \ID {I_{\C\otimes E}}
\newcommand \J[1] {J_{#1,x_0}}
\newcommand \QQ[1] {\Q_{#1,x_0}}
\newcommand \KKKK[4]{\KK{\KK{#1}{#2}}{\KK{#3}{#4}}}
\newcommand \KKk[3]{\K\big[\K\big[#1,#2\big],#3\big]}

\newcommand \R {\mathbb R}

\newcommand \C {\mathbb C}

\newcommand \N {\mathbb N}

\newcommand \Z {\mathbb Z}

\newcommand \ignore[1] {}

\newtheorem{prop}{Proposition}[section]
\newtheorem{theo}[prop]{Theorem}

\newtheorem{lem}[prop]{Lemma}

\theoremstyle{definition}
\newtheorem*{ackn*}{Acknowledgements}

\theoremstyle{remark}

\numberwithin{equation}{section}
\labelformat{theo}{Theorem~#1}
\labelformat{lem}{Lemma~#1}
\labelformat{cor}{Corolary~#1}
\labelformat{ex}{Example~#1}
\labelformat{defi}{Definition~#1}
\labelformat{prop}{Proposition~#1}
\labelformat{section}{Section~#1}

\begin{document}

\title{\bf{On the composition of Berezin-Toeplitz operators on symplectic manifolds}}
\author{Louis IOOS}
\date{16 mars 2017}

\maketitle

\begin{abstract}

We compute the second coefficient of the composition of two Berezin-Toeplitz operators associated with the $\text{spin}^c$ Dirac operator on a symplectic manifold, making use of the full off-diagonal expansion of the Bergman kernel.

\end{abstract}

\section{Introduction}

In \cite{MM08b}, Ma and Marinescu studied in detail Berezin-Toeplitz quantization for symplectic manifolds, introducing kernel calculus as a method to compute the coefficients of the asymptotic expansion of the associated Toeplitz operators. They considered the following situation: let $(X,\om)$ be a compact symplectic manifold of dimension $2n$, and $(E,h^E),\ (L,h^L)$ be Hermitian vector bundles on $X$ with $\rk(L)=1$, endowed with Hermitian connections $\nabla^E,\ \nabla^L$. If $R^L$ denotes the curvature of $\nabla^L$, we assume the following so-called \emph{prequantization condition}:
\begin{equation}
\label{prequantization}
\om=\frac{\sqrt{-1}}{2\pi}R^L.
\end{equation}

Let $J\in\End(TX)$ be an almost complex structure on $TX$ compatible with $\om$, and let $g^{TX}$ be the Riemannian metric on $TX$ defined by
\begin{equation}
\label{Jgras}
\om(u,v)=g^{TX}(Ju,v),
\end{equation}
for any $u ,v\in TX$. We denote by $L^p$ the $p^{th}$ tensor power of $L$ and $D_p$ the \emph{$spin^c$ Dirac operator} acting on the smooth sections of $E_p:=\Spin\otimes L^p\otimes E$. The metrics $g^{TX}, h^L$ and $h^E$ induce the usual $L^2$-scalar product on the space $L^2(X,E_p)$ of the square integrable sections of $E_p$. The orthogonal projection of $L^2(X,E_p)$ on $\Ker(D_p)$ with respect to this product is denoted by $P_p$ and is called the \emph{Bergman projection}. For $f\in\cinf(X,\End(E))$, the \emph{Berezin-Toeplitz quantization} of $f$ is the family $\{T_{f,p}\}_{p\in\N}$ of operators acting on $L^2(X,E_p)$ by
\begin{equation}
\label{Berezin-Toeplitz quantization}
T_{f,p}=P_pfP_p:L^2(X,E_p)\rightarrow L^2(X,E_p),
\end{equation}
where $f$ denotes the operator acting by pointwise multiplication by $f$. 

More generally, a family $\{T_p\}_{p\in\N}$ of bounded operators acting on $L^2(X,E_p)$ is called a \emph{Toeplitz operator} if $P_pT_pP_p=T_p$ for all $p\in\N$, and if there exists a sequence of sections $g_r\in\cinf(X,\End(E))$ for all $r\in\N$ such that
\begin{equation}\label{deftoep}
T_p=\sum_{r=0}^\infty T_{g_r,p} p^{-r} + O(p^{-\infty}),
\end{equation}
where, for all $r\in\N$, the family of operators $\{T_{g_r,p}\}_{p\in\N}$ is the Berezin-Toeplitz quantization of $g_r$ in the sense of \eqref{Berezin-Toeplitz quantization}. Here the notation $O(p^{-\infty})$ means that, for all $k\in\N$, the sum up to order $k$ is a $O(p^{-k})$ of the left member for the operator norm.

In \cite[Th.1.1]{MM08b}, Ma and Marinescu proved that the set of Toeplitz operators as defined in \eqref{deftoep} forms an algebra. More precisely, given $f, g\in\cinf(X,\End(E))$, they established that
\begin{equation}
\label{compo}
T_{f,p}T_{g,p}=\sum\limits_{r=0}^\infty p^{-r}T_{C_r(f,g),p} + O(p^{-\infty}),
\end{equation}
where $C_r$ are bidifferential operators, with $C_0(f,g)=fg$. In particular, we get the following formula:
\begin{equation}\label{Toeplitzfle1}
T_{f,p}T_{g,p}=T_{fg,p}+O(p^{-1}),
\end{equation}
which shows that the composition of two Toeplitz operators approach the usual pointwise composition of endomorphisms in the semi-classical limit, when $p$ tends to $\infty$. Moreover, in the case $f, g\in\cinf(X)$, they showed that $C_1(f,g)-C_1(g,f)=\sqrt{-1}\{f,g\}$, where $\{.,.\}$ denotes the Poisson bracket associated to the symplectic form $2\pi\om$. We thus get the following formula:
\begin{equation}\label{Toeplitzfle2}
[T_{f,p},T_{g,p}]=p^{-1}T_{\{f,g\},p}+O(p^{-2}),
\end{equation}
which shows that the family $\{T_{f,p}\}_{p\in\N}$ indeed satisfies the expected semi-classical limit for a quantization.

Especially interesting is the case of $J$ coming from a complex structure, making $X$ into a Kähler manifold. In this case, we ask $(E,h^E), (L,h^L)$ to be holomorphic Hermitian vector bundles, and $\nabla^E, \nabla^L$ to be the associated holomorphic Hermitian connections. The $spin^c$ Dirac operator $D_p$ is then given by
\begin{equation}
D_p=\sqrt{2}(\dbar^{L^p\otimes E} +\dbar^{L^p\otimes E,*}),
\end{equation}
where $\dbar^{L^p\otimes E}$ denotes the holomorphic $\dbar$-operator on $L^p\otimes E$ acting on the Dolbeault complex $\oplus_q\Omega^{0,q}(X,L^p\otimes E)=\cinf(X,E_p)$, and $\dbar^{L^p\otimes E,*}$ its formal adjoint for the $L^2$-scalar product. By Hodge theory, we get
\begin{equation}\label{Hodge}
\Ker(D_p|_{\Om^{0,q}(X,L^p\otimes E)})\simeq H^q(X,L^p\otimes E),
\end{equation}
where $H^q(X,L^p\otimes E)$ denotes the $q^{th}$ Dolbeault cohomology group associated to $L^p\otimes E$. The prequantization condition \eqref{prequantization} implying $L$ positive, by the Kodaira-Serre vanishing theorem we get for any $q>0$,
\begin{equation}
\label{Kodaira-Serre}
H^q(X,L^p\otimes E)=0,
\end{equation}
whenever $p$ is sufficiently large. Picking such a $p$, the identification \eqref{Hodge} together with \eqref{Kodaira-Serre} imply $\Ker(D_p)\simeq H^0(X,L^p\otimes E)$, which gives back the usual setting of geometric quantization on Kähler manifolds, the space $H^0(X,L^p\otimes E)$ being the space of holomorphic sections of $L^p\otimes E$. In the general symplectic setting however, Dolbeault cohomology doesn't exist, and $\Ker(D_p)$ is then a natural generalization of the space $H^0(X,L^p\otimes E)$.

The theory of Berezin-Toeplitz quantization in the Kähler case for $E=\C$ has first been developed by Bordemann, Meinreken and Schlichenmaier in \citep{BMS94} and Schlichenmaier in \citep{Sch00}. Their approach is based on the work of Boutet de Monvel and Sjöstrand on the Szegö kernel in \citep{BS75}, and the theory of Toeplitz structures developed by Boutet de Monvel and Guillemin in \citep{BG81} (see also \citep{Cha03}).

In the Kähler case, the data given in \eqref{Berezin-Toeplitz quantization} can be computed much more explicitly, and in \citep[Th.0.3]{MM12}, Ma and Marinescu gave the following formula for the second coefficient $C_1(f,g)$ for $f,g\in\cinf(X,\End(E))$:
\begin{equation}\label{fleC_1(f,g)}
C_1(f,g)=-\frac{1}{2\pi}\langle\nabla^{1,0}f,\nabla^{0,1}g\rangle,
\end{equation}
where $\nabla^{1,0}$ and $\nabla^{0,1}$ denote the holomorphic and anti-holomorphic part of the connection on $\End(E)$ induced by $\nabla^E$, and $\<.,.\>$ denotes the pairing induced by $g^{TX}$ on $T^*X\otimes \End(E)$ with values in $\End(E)$. The formula \eqref{fleC_1(f,g)} is compatible with the following description of the Poisson bracket in the case $E=\C$:
\begin{equation}\label{explicit Poisson}
\sqrt{-1}\{f,g\}=-\frac{1}{2\pi}\left(\langle\nabla^{1,0}f,\nabla^{0,1}g\rangle-\langle\nabla^{1,0}g,\nabla^{0,1}f\rangle\right).
\end{equation}

Ma and Marinescu also computed the coefficient $C_2(f,g)$ for $f, g\in\cinf(X)$, and gave in \citep[Th.0.3]{MM12} formulas for the first coefficients of the expansion in $p\in\N$ of the kernel of a Berezin-Toeplitz operator on the diagonal. These formulas have been used for the study of canonical metrics via balanced embeddings by Fine in \citep[Th.10]{Fin12} for the quantization of the Lichnerowicz operator, and then by Keller, Meyer and Seyyedali in  \citep[Prop.3.6]{KMS16} for the quantization of the Laplacian operator on vector bundles (see also \citep{DV15} on this last topic). In \citep[Th.1.4, Th.1.5]{Hsi12}, Hsiao gave a new proof of \citep[Th.0.2, Th.0.3]{MM12} for $E=\C$ using results from microlocal analysis of \citep{BS75}.

In the context of deformation quantization, the properties \eqref{Toeplitzfle1} and \eqref{Toeplitzfle2} in the case $E=\C$ imply that the expansion \eqref{compo} defines a star product on $\cinf(X)$, called the \emph{Berezin-Toeplitz star product} (see for instance \citep[Rem.7.4.2]{MM07} and \citep {Xu13}).

In this paper, we use methods developed in \cite{MM08b} as well as results of \cite{MM06} in order to compute $C_1(f,g)$ for $f,g\in\cinf(X,\End(E))$ in the general symplectic case described in the beginning of this section. Analogous to \eqref{fleC_1(f,g)}, our result is:

\begin{theo}\label{goal}
Let $(X,\om)$ be a compact symplectic manifold equiped with a Hermitian line bundle with Hermitian connection $(L,h^L,\nabla^L)$ satisfying the prequantization condition \eqref{prequantization}, and let $(E,h^E,\nabla^E)$ be a Hermitian vector bundle with Hermitian connection. Let $J\in\End(TX)$ be an almost complex structure on $X$, and $g^{TX}$ be the Riemannian metric defined by \eqref{Jgras}. For any $f,g\in\cinf(X,\End(E))$, the second coefficient of the asymptotic expansion of $T_{f,p}T_{g,p}$ as in \eqref{compo} is 
\begin{equation}
C_1(f,g)=-\frac{1}{2\pi}\langle\nabla^{1,0}f,\nabla^{0,1}g\rangle,
\end{equation}
where $\nabla^{1,0}$ and $\nabla^{0,1}$ denote the holomorphic and anti-holomorphic part of the connection on $\End(E)$ induced by $\nabla^E$, and $\<.,.\>$ denotes the pairing induced by $g^{TX}$ on $T^*X\otimes \End(E)$ with values in $\End(E)$.
\end{theo}

In the case $E=\C$, Charles treated the theory of Berezin-Toeplitz quantization for symplectic manifolds in \citep{Cha16a} using microlocal analysis of \citep{BG81}, and computed the coefficient $C_1(f,g)$ for $f, g\in\cinf(X)$ in \citep{CL}.

\begin{ackn*}
This paper is a part of the author's PhD thesis under the supervision of Pr. Xiaonan Ma. The author wants to thank his supervisor for his support and for the numerous advices on the redaction of this paper. This work was supported by grants from Région ile-de-France.
\end{ackn*}

\section{The local model for Toeplitz operators}\label{local model}

This section is dedicated to set the context and the notations, and to describe the local model which will be used in the next section for the computations.

\subsection{Setting}

Let $(X,\om)$ be a compact symplectic manifold, endowed with an almost complex structure $J$ compatible with $\om$ on its tangent bundle $TX$. We denote by $g^{TX}$ the Riemannian metric defined by \eqref{Jgras}, and by $\nabla^{TX}$ the associated Levi-Civita connection on $TX$. The almost complex structure $J$ induces a splitting
\begin{equation}\label{splitJ}
TX\otimes\C=T^{(1,0)}X\oplus T^{(0,1)}X
\end{equation}
on the complexification of the tangent bundle into the eigenspaces of $J$ correponding to the eigenvalues $\sqrt{-1}$ and $-\sqrt{-1}$ respectively. This allows us to define the total exterior product bundle $\Spin$, which is actually a Clifford bundle: for any $v\in TX$ with decomposition $v=v_{1,0}+v_{0,1}$ according to \eqref{splitJ}, we define the Clifford action of $v$ on $\Spin$ by
\begin{equation}
c(v)=\sqrt{2}(v^*_{1,0}-i_{v_{0,1}}),
\end{equation}
where $v^*_{1,0}$ denotes the wedge product by the metric dual of $v_{1,0}$ in $T^{*(0,1)}X$, and $i_{v_{0,1}}$ denotes the contraction by $v_{0,1}\in T^{(0,1)}X$. Let $\nabla^{\det}$ be the connection on $\det(T^{(1,0)}X)$ induced by the natural projection of $\nabla^{TX}$ on $T^{(1,0)}X$ via the decomposition \eqref{splitJ}. We denote by $\nabla^{\Cl}$ the Clifford connection on $\Spin$  induced by $\nabla^{TX}$ and $\nabla^{\det}$ as defined in \cite[\S 1.3.1]{MM07}.
%

Recall now that $\nabla^E$ and $\nabla^L$ are Hermitian connections on the Hermitian vector bundle $(E,h^E)$ and the Hermitian line bundle $(L,h^L)$ respectively. The vector bundle
\begin{equation}
E_p=L^p\otimes\Spin\otimes E
\end{equation}
is naturally endowed with the Hermitian product induced by $g^{TX}, h^L$ and $h^E$. Then there is a natural $L^2$-scalar product on $\cinf(X,E_p)$ induced by the Hermitian product of $E_p$ and the Riemannian volume form $dv_X$ on $X$ associated to $g^{TX}$. We denote by $\nabla^{E_p}$ the connection on $E_p$ induced by $\nabla^{\Cl}, \nabla^L$ and $\nabla^E$. We define then the $spin^c$ Dirac operator $D_p$ locally by
\begin{equation}\label{spin^c}
D_p=\sum\limits_{j=1}^{2n} c(e_j)\nabla^{E_p}_{e_j}:\cinf(X,E_p)\rightarrow\cinf(X,E_p),
\end{equation}
where $\{e_j\}_{j=1}^{2n}$ is any local orthonormal frame of $TX$ with respect to $g^{TX}$. Then $D_p$ is a formally self-adjoint operator on $\cinf(X,E_p)$ with respect to the $L^2$-scalar product.

\subsection{Model operator}

Let us fix a point $x_0\in X$, and let us choose $\epsilon_0>0$ so that the exponential map at $x_0$ induces a diffeomorphism between the geodesic ball $B^X(x_0,\epsilon_0)\subset X$ and the open ball $B(0,\epsilon_0)\subset T_{x_0}X$, where $T_{x_0}X$ is endowed with the Euclidean metric induced by $g^{TX}$. We trivialize $L, E$ and $E_p$ over $B(0,\epsilon_0)$ by identification with their respective fibre at $x_0$ through parallel transport with respect to their respective connections along geodesics. We then identify $L_{x_0}$ with $\C$ choosing a unit vector. Let us note that as $\End(L^p_{x_0})$ is canonically identified with $\C$, our results will not depend on this choice. 

Let $\{w_j\}_{j=1}^n$ be an orthonormal basis of $T_{x_0}^{(1,0)}X$ with respect to the Hermitian product induced by $g^{TX}$. This induces a basis $\{e_j\}_{j=1}^{2n}$ of $T_{x_0}X$, orthonormal with respect to the Euclidean product induced by $g^{TX}$, such that
\begin{equation}
\label{base}
e_{2j-1}=\frac{1}{\sqrt{2}}(w_j+\overline{w}_j)\ \text{ and }\ e_{2j}=\frac{\sqrt{-1}}{\sqrt{2}}(w_j-\overline{w}_j),
\end{equation}
for any $1\leq j\leq n$. We use this basis to identify $T_{x_0}X$ with $\R^{2n}$. We denote by $Z=(Z_1,\dots,Z_{2n})$ the induced real coordinates, and by $z=(z_1,\dots,z_n)$ the complex coordinates on $\C^n$ such that $z_i=Z_{2i-1}+\sqrt{-1}Z_{2i}$ for any $1\leq i \leq n$. We thus have $e_j=\partial/\partial Z_j$ for any $1\leq j\leq 2n$, and the following equality of vector fields holds:
\begin{equation}\label{Z=z+zbar}
\sum_{i=1}^{2n}Z_i\frac{\partial}{\partial Z_i}=\sum_{j=1}^{2n} \left(z_j\frac{\partial}{\partial z_j}+\bar{z}_j\frac{\partial}{\partial \bar{z}_j}\right),
\end{equation}
which is not to be confused with the equality of coordinates $Z=(z+\bar{z})/2$. We write $dZ$ for the canonical Lebesgue measure of $\R^{2n}$ with respect to the variable $Z$.

Let us now consider the Hilbert space $L^2(\R^{2n},(\Spin\otimes E)_{x_0})$ with the $L^2$-scalar product induced by the Hermitian product on $(\Spin\otimes E)_{x_0}$ and the Lebesgue measure of $\R^{2n}$. By the identification $T_{x_0}X\cong\R^{2n}$ and the trivializations above, we identify sections in $L^2(X,E_p)$ with sufficiently small compact support around $x_0$ with functions in $L^2(\R^{2n},(\Spin\otimes E)_{x_0})$. 

It is shown in the complex case in \cite[Chap.4]{MM07} and generalized to the symplectic case in \cite[Chap.8]{MM07} how the operator $(D_p)^2$ is approximated for large $p$, after a convenient rescaling in $\sqrt{p}$, by an operator $\LL^0_2$ acting on $\cinf(\R^{2n},(\Spin\otimes E)_{x_0})$ defined by the formulas
\begin{equation}
\label{model operator 1}
\begin{split}
\LL^0_2=\LL+4\pi\overline{w}^j i_{\overline{w}_j}, \quad & \quad \LL=\sum\limits_{j=1}^{n}b_jb_j^+,\\
 & \\
b_i=-2\frac{\partial}{\partial z_i}+\pi \bar{z}_i, \quad & \quad b_i^+=2 \frac{\partial}{\partial{\bar{z}_i}}+\pi z_i,
\end{split}
\end{equation}
for any $1\leq i\leq n$, where $z_i$ and $\bar{z}_i$ denote the scalar multiplication on $(\Spin\otimes E)_{x_0}$ by $z_i$ and $\bar{z}_i$ respectively. The differential operator $\LL$ acts on the scalar part of smooth functions with values in $(\Spin\otimes E)_{x_0}$, and can thus be seen as a differential operator on $\cinf(\R^{2n})$, still denoted by $\LL$. We call $\LL$ the \emph{model operator}. It is a densely defined self-adjoint operator on $L^2(\R^{2n})$ and has the following spectral properties:
\begin{prop}{(\cite[\S 4.1.20]{MM07})}\label{spectrumL}
The spectrum of $\LL$ on $L^2(\R^{2n})$ is given by
\begin{equation}\label{specalpha}
\Spec(\LL)=\Big\{4\pi\sum\limits_{i=1}^n \alpha_i\ \Big|\ \alpha=(\alpha_1,\dots,\alpha_n)\in\N^n\Big\},
\end{equation}
and an orthogonal basis of the eigenspace indexed by $\alpha\in\N^n$ as in \eqref{specalpha} is given by
\begin{equation}
b^\alpha\left(z^\beta \exp\left(-\frac{\pi}{2}\sum\limits_{i=1}^n |z_i|^2\right)\right)\ \text{for any}\ \beta\in\N^n.
\end{equation}
\end{prop}

The corresponding orthogonal projection $\PP:L^2(\R^{2n})\rightarrow\Ker(\LL)$ has smooth kernel with respect to $dZ$, which for all $Z, Z'\in\R^{2n}$ is easily computed to be
\begin{equation}
\label{PP=1}
\PP(Z,Z')=\exp\left(-\frac{\pi}{2}\sum\limits_{i=1}^n (|z_i|^2+|z_i'|^2-2z_i\bar{z}_i')\right).
\end{equation}

As the operator $\LL$ in \eqref{model operator 1} acts only on the scalar part of functions with values in $(\Spin\otimes E)_{x_0}$, the kernel of the associated projection
\begin{equation}
\PP:L^2(\R^{2n},(\Spin\otimes E)_{x_0}))\rightarrow\Ker(\LL)
\end{equation}
acts on $(\Spin\otimes E)_{x_0}$ by scalar multiplication and is still given by \eqref{PP=1}. 

The space $(\Spin\otimes E)_{x_0}$ has a natural $\Z$-graduation given by the one on $\Spin$, and we denote by $(\C\otimes E)_{x_0}$ the degree $0$ subspace of $(\Spin\otimes E)_{x_0}$ for this graduation. We write
\begin{equation}\label{ID}
\ID:(\Spin\otimes E)_{x_0}\rightarrow (\C\otimes E)_{x_0}
\end{equation}
for the natural projection, which is orthogonal with respect to the Hermitian product. It induces a projection on the $L^2$-sections, still denoted by $\ID$, which is then orthogonal with respect to the $L^2$-scalar product.

The two terms defining $\LL^0_2$ in \eqref{model operator 1} are positive commuting operators, and the orthogonal projections on their kernels in $L^2(\R^{2n},(\Spin\otimes E)_{x_0})$ are given by $\PP$ and $\ID$ respectively. Consequently, the orthogonal projection on the kernel of $\LL^0_2$, seen as a densely defined self-adjoint operator acting on $L^2(\R^{2n},(\Spin\otimes E)_{x_0})$, is given by
\begin{equation}
\label{def P}
P=\PP\ID:L^2(\R^{2n},(\Spin\otimes E)_{x_0})\rightarrow\Ker(\LL^0_2).
\end{equation}

We denote by $\Ker(\LL^0_2)^{\perp}$ the orthogonal space of $\Ker(\LL^0_2)$ in $L^2(\R^{2n},(\Spin\otimes E)_{x_0})$, and by $P^{\perp}$ the associated orthogonal projection. Using \ref{spectrumL} and \eqref{model operator 1}, it is easy to compute explicitly the inverse of $\LL^0_2$ on $\Ker(\LL^0_2)^{\perp}$ by inverting its eigenvalues. It thus makes sense to write
\begin{equation}
\label{def L^{-1}}
(\LL^0_2)^{-1}P^{\perp}:L^2(\R^{2n},(\Spin\otimes E)_{x_0})\rightarrow L^2(\R^{2n},(\Spin\otimes E)_{x_0}).
\end{equation}

\subsection{Kernel calculus}\label{kernel calculus section}

We introduce now the kernel calculus on $\C^{n}$ developed by Ma and Marinescu in \citep{MM08b}, which will be the basis for our calculations in the next section.

If $T$ is a bounded operator on $L^2(\R^{2n},(\Spin\otimes E)_{x_0})$ with smooth kernel with respect to $dZ$, we will denote its evaluation at $Z, Z'\in\R^{2n}$ by
\begin{equation}
T(Z,Z')\in\End(\Spin\otimes E)_{x_0}.
\end{equation}

If $F(Z,Z')\in\End(\Spin\otimes E)_{x_0}$ is a polynomial in $Z,Z'\in\R^{2n}$, we denote by $F\PP$ the operator on $L^2(\R^{2n},(\Spin\otimes E)_{x_0})$ defined by the kernel $F(Z,Z')\PP(Z,Z')$, so that
\begin{equation}\label{FPP}
(F\PP)(Z,Z')=F(Z,Z')\PP(Z,Z'),
\end{equation}
for all $Z,Z'\in\R^{2n}$. By the explicit expression of $\PP(Z,Z')$ in \eqref{PP=1}, the formula \eqref{FPP} defines in fact a bounded operator on $L^2(\R^{2n},(\Spin\otimes E)_{x_0})$. 

Using these notations, we can state the following result, which comes essentially from \citep[\S 2]{MM08b}.

\begin{prop}\label{K}
For any $Q(Z,Z')$ and $F(Z,Z')\in\End(\Spin\otimes E)_{x_0}$, polynomials in $Z,Z'\in\R^{2n}$, there exists $\KK{F}{Q}(Z,Z')\in\End(\Spin\otimes E)_{x_0}$, polynomial in $Z,Z'\in\R^{2n}$, such that

\begin{equation}
\label{def KK}
\KK{F}{Q}\PP=(F\PP)(Q\PP),
\end{equation}

where the left hand side denotes the composition of two operators defined by their kernels as in \eqref{FPP}.

Furthermore, for all $F(Z,Z'),\ G(Z,Z')$ and $H(Z,Z')\in\End(\Spin\otimes E)_{x_0}$, polynomials in $Z, Z'\in\R^{2n}$, the following formulas hold:

\begin{equation}
\label{K1}
\KKK{F}{G}{H}=\KKk{F}{G}{H},
\end{equation}

\begin{equation}
\label{K2}
\KK{\ID}{\ID}=\ID.
\end{equation}

For any $q(Z)$ polynomial in $Z$ with scalar values,

\begin{equation}
\label{K3}
\KK{F}{q(Z)G}=\KK{q(Z')F}{G}.
\end{equation}

For any $Q(Z)$ polynomial in $Z$ with values in $\End(E_{x_0})$,

\begin{equation}
\label{K4}
\begin{split}
\KK{Q(Z)F}{G}=Q(Z)\KK{F}{G},\\
 \\
\KK{F}{GQ(Z')}=\KK{F}{G}Q(Z').
\end{split}
\end{equation}

Finally, for $A\in\End(E_{x_0})$ and $G(Z,Z')\in\End(\Spin\otimes E)_{x_0}$ polynomial in $Z,Z'\in\R^{2n}$ commuting with $A$, we have:

\begin{equation}
\label{K5}
\begin{split}
& A\KK{G}{F}=\KK{GA}{F}=\KK{G}{AF},\\
 \\
& \KK{FA}{G}=\KK{F}{AG}=\KK{F}{G}A.
\end{split}
\end{equation}
\end{prop}
\begin{proof}
Let us first deal with the case $F=1$. The kernel of the composition $\PP(G\PP)$ is given by $\PP(G(Z,Z')\PP(Z,Z'))$, where the operator $\PP$ acts on the variable $Z\in\R^{2n}$. Thus the variable $Z'\in\R^{2n}$ acts as a parameter in this situation, and we are reduced to the case $G(Z,Z')=G(Z)$ not depending on $Z'$. Using \ref{spectrumL}, this can be computed by induction on the degree of $G$ in $z, \bar{z}\in\C^n$:

First, by \ref{spectrumL} and \eqref{PP=1}, we get
\begin{equation}\label{zbeta}
\PP(z^\beta\PP(Z,Z'))=z^\beta\PP(Z,Z'),
\end{equation}
for any $\beta\in\N^n$. Next, let us notice that by \eqref{model operator 1} and \eqref{def P},
\begin{equation}\label{zbar}
b_i\PP(Z,Z')=2\pi(\bar{z}_i-\bar{z}_i')\PP(Z,Z').
\end{equation}

Here, $b_i$ defined in \eqref{model operator 1} acts on the $Z$ variable. As $b_i\PP(Z,Z')$ is in the orthogonal of $\Ker(\LL)$ by \ref{spectrumL}, we get using \eqref{zbeta},
\begin{equation}\label{Pzi=Pzi'bar}
\PP(\bar{z}_i\PP(Z,Z'))=\bar{z}_i'\PP(Z,Z').
\end{equation}

Now by the definition of $b_i$ in \eqref{model operator 1}, we know that
\begin{equation}\label{bicom}
[G(Z),b_i]=2\frac{\partial}{\partial z_i}G(Z),
\end{equation}
for any $G(Z)\in\End(\Spin\otimes E)_{x_0}$ polynomial in $Z\in\R^{2n}$. Using the fact that $b^\alpha z^\beta\PP(Z,Z')$ is in $\Ker(\LL)^{\perp}$ by \ref{spectrumL}, we can compute the case of general $F$ by induction through repeated applications of \eqref{zbar} and \eqref{bicom}.

Now for the general case, if $F_1(Z)$ and $F_2(Z')\in\End(\Spin\otimes E)_{x_0}$ are two polynomials in $Z$ and $Z'\in\R^{2n}$ respectively, by the definition of operators associated to kernels as in \eqref{FPP}, we have the following two easy facts:
\begin{equation}\label{K3K4pre}
\begin{split}
(F_1(Z)G\PP)(Z,Z') & =F_1(Z)(G\PP)(Z,Z'),\\
 \\
(GF_2(Z')\PP)(Z,Z') & =(G\PP)(Z,Z')F_2(Z'),
\end{split}
\end{equation}
where we use for the second equality the fact that $\PP(Z,Z')$ has scalar values by \eqref{PP=1}. By \eqref{def KK} and the usual formula for composition of kernels, recall that
\begin{equation}\label{kernelcompo}
\left(\KK{F}{G}\PP\right)(Z,Z')=\int_{\R^{2n}} F(Z,Z'')\PP(Z,Z'')G(Z'',Z')\PP(Z'',Z') dZ''.
\end{equation}

Then from \eqref{K3K4pre} and \eqref{kernelcompo}, for any $F(Z,Z')\in\End(\Spin\otimes E)_{x_0}$ polynomial in $Z,Z'\in\R^{2n}$,
\begin{equation}\label{K3K4}
\begin{split}
\left(\KK{F_1(Z)F}{G}\PP\right)(Z,Z') & =F_1(Z)\left(\KK{F}{G}\PP\right)(Z,Z'),\\
 \\
\left(\KK{FF_2(Z')}{G}\PP\right)(Z,Z') & =\left(\KK{F}{F_2(Z)G}\PP\right)(Z,Z'),
\end{split}
\end{equation}
so the general case reduces to the previous one. As a byproduct of \eqref{K3K4pre} and \eqref{K3K4}, we get \eqref{K3} and \eqref{K4} as well.

The associativity \eqref{K1} is obvious from \eqref{def KK}. As $\PP(Z,Z')$ commutes with $\ID$ by \eqref{PP=1}, we get \eqref{K2}. Finally, \eqref{K3K4pre} and \eqref{K3K4} applied to $A\in\End(E_{x_0})$ constant and commuting with $G$ gives \eqref{K5}.
\end{proof}

\ref{K}, together with its proof, is at the basis of the computations in this paper. As an application, we compute the following special cases of the kernel calculus for any $1\leq i,j\leq n$, which will be used constantly in the forthcoming computations:
\begin{equation}
\label{kernel calculus}
\begin{split}
& \KK{\ID}{\bar{z}_j\ID}=\bar{z}_j'\ID,\quad \KK{\ID}{z_j\ID}=z_j\ID,\\
& \KK{z_i\ID}{\bar{z}_j\ID}=z_i\bar{z}_j'\ID,\quad \KK{\bar{z}_i\ID}{z_j\ID}=\bar{z}_iz_j\ID,\\
& \KK{z_i'\ID}{\bar{z}_j\ID}=\frac{1}{\pi}\delta_{ij}\ID+z_i\bar{z}_j'\ID,\\
& \KK{\bar{z}_i'\ID}{z_j\ID}=\frac{1}{\pi}\delta_{ij}\ID+\bar{z}_i'z_j\ID,\\
& \KK{\ID}{\bar{z}_iz_j\ID}=\frac{1}{\pi}\delta_{ij}\ID+\bar{z}_i'z_j\ID,\\
& \KK{\ID}{z_i'\ID}=z_i'\ID,\quad \KK{\ID}{\bar{z}_i'\ID}=\bar{z}_i'\ID,\\
& \KK{\ID}{z_iz_j\ID}=z_iz_j\ID,\quad \KK{\ID}{\bar{z}_i\bar{z}_j\ID}=\bar{z}_i'\bar{z}_j'\ID.
\end{split}
\end{equation}


\section{Berezin-Toeplitz quantization and Bergman kernel}

In this section, we recall the results of \citep{MM06}, \citep{MM07} and \citep{MM08b} on the reduction of Berezin-Toeplitz quantization to the local model described in \ref{local model}, and then go on to the computation of the second coefficient of the asymptotic expansion. The general reference for the background on the theory is \citep{MM07}.

\subsection{The asymptotic expansion of Toeplitz operators}

For any $p\in\N$, let $T_p$ be an operator acting on $L^2(X,E_p)$ with smooth kernel $T_p(x,x')$ with respect to $dv_X$ at $x, x'\in X$. For all $x_0\in X$, it induces
\begin{equation}
T_{p,x_0}(Z,Z')\in\End(\Spin\otimes E)_{x_0}
\end{equation}
through the trivializations given in \ref{local model}, where $Z, Z'\in B(0,\epsilon_0)\subset T_{x_0}X\simeq\R^{2n}$ are the respective images of $x, x'\in B^X(x_0,\epsilon_0)\subset X$ in the exponential coordinates. 

To estimate the kernels of any family $\{T_p\}_{p\in\N}$ of operators acting on $L^2(X,(\Spin\otimes E)_{x_0})$, we will use the following notation given in \cite[Not.4.4]{MM08b}: we write
\begin{equation}\label{cong}
p^{-n}T_{p,x_0}(Z,Z')\cong\sum\limits_{r=0}^{\infty}(\QQ{r}\PP)(\sqrt{p}Z,\sqrt{p}Z')p^{-r/2} + O(p^{-\infty}),
\end{equation}
where $\{\QQ{r}(Z,Z')\}_{r\in\N}$ is a family of polynomials in $Z, Z'\in\R^{2n}$ with values in $\End(\Spin\otimes E)_{x_0}$ and depending smoothly on $x_0\in X$, if for any $k\in\N$, there is $\epsilon>0, C_0>0$ such that for any $l\in\N$, there exist $C>0, M\in\N$, such that for $|Z|, |Z'|<\epsilon$, the following estimate holds:
\begin{equation}\label{defcong}
\begin{split}
\left|p^{-n}\right. & T_{p,x_0}(Z,Z')\kappa_{x_0}^{1/2}(Z)\kappa_{x_0}^{1/2}(Z')-\sum\limits_{r=0}^k(\QQ{r}\PP)(\sqrt{p}Z,\sqrt{p}Z')\left. p^{-r/2}\right|_{\CC^l}\\
&\leq Cp^{-\frac{k+1}{2}}(1+\sqrt{p}|Z|+\sqrt{p}|Z'|)^M\exp\left(-C_0\sqrt{p}|Z-Z'|\right)+O(p^{-\infty}).
\end{split}
\end{equation}

Here $|.|_{\CC^l}$ denotes the $\CC^l$ norm with respect to $x_0\in X$ with respect to the induced connection on the pullback bundle $\pi^*(\End(\Spin\otimes E))$ over $TX\times_X TX$, where $\pi:TX\times_X TX\rightarrow X$ is the fibre product of $TX$ with itself over $X$. In the same way, when we say that a polynomial in $Z, Z'\in\R^{2n}$ with values in $\End(\Spin\otimes E)_{x_0}$ depends smoothly on $x_0\in X$, it is in that sense.

The function $\kappa_{x_0}\in\cinf(B(0,\epsilon_0))$ is defined for $Z\in B(0,\epsilon_0)\subset\R^{2n}$ by
\begin{equation}
dv_X(Z)=\kappa_{x_0}(Z)dZ.
\end{equation}

Its appearance in the formula \eqref{defcong} is necessary to make the comparison between kernels consistent. Note that $\kappa_{x_0}(0)=1$.

We will apply the notation \eqref{defcong} to estimate the kernel of the Bergman projection $P_p$ on $\Ker(D_p)$ as defined in the introduction, namely through the following \emph{off-diagonal expansion} of the Bergman kernel:

\begin{theo}{(\citep[Th.4.18']{DLM06})}
\label{offdiag}
There exists a family $\{\J{r}\}_{r\in\N}$ of polynomials in $Z, Z'\in\R^{2n}$ with values in $\End(\Spin\otimes E)_{x_0}$ and depending smoothly on $x_0\in X$ such that the following expansion holds in the sense of \eqref{defcong}:

\begin{equation}
\label{exp J}
p^{-n}P_{p,x_0}(Z,Z')\cong\sum\limits_{r=0}^{\infty}(\J{r}\PP)(\sqrt{p}Z,\sqrt{p}Z')p^{-r/2} + O(p^{-\infty}).
\end{equation}

\end{theo}

\ref{offdiag} gives a strong control on the Bergman kernel outside the diagonal, and is used in \citep{MM08b} to prove the following result:

\begin{theo}{(\citep[Th.1.1]{MM08b})}
\label{expansion}
Let $f$ and $g\in\cinf(X,\End(E))$. There exist families of polynomials in $Z, Z'\in\R^{2n}$ with values in $\End(\Spin\otimes E)_{x_0}$ and depending smoothly on $x_0\in X$, respectively denoted by $\{\QQ{r}(f)\}_{r\in\N}$ and $\{\QQ{r}(f,g)\}_{r\in\N}$, such that the following expansions hold in the sense of \eqref{cong}:

\begin{equation}
\label{exp Q}
p^{-n}T_{f,p,x_0}(Z,Z')\cong\sum\limits_{r=0}^{\infty}(\QQ{r}(f)\PP)(\sqrt{p}Z,\sqrt{p}Z')p^{-r/2} + O(p^{-\infty}),
\end{equation}

\begin{equation}
\label{exp Q2}
p^{-n}(T_{f,p}T_{g,p})_{x_0}(Z,Z')\cong\sum\limits_{r=0}^{\infty}(\QQ{r}(f,g)\PP)(\sqrt{p}Z,\sqrt{p}Z')p^{-r/2} + O(p^{-\infty}).
\end{equation}
\end{theo}
%
%

Recall that in \ref{local model}, we defined an operator $\LL^0_2$ aproximating $D_p$ for large $p$, after a convenient rescaling in $\sqrt{p}$. We can refine this by the following (see \citep[\S 4.1.3]{MM07} for a precise statement): after a convenient rescaling in $\sqrt{p}=:1/t$, the restriction of $D_p$ on $B^X(x_0,\epsilon_0)$ is equal, through the trivializations of \ref{local model}, to an operator $\LL^t_2$ on $B(0,\epsilon_0)$ satisfying
\begin{equation}\label{expLt2}
\LL^t_2=\LL^0_2+\sum_{r=1}^mt^r \OO_r +O(t^{m+1}),
\end{equation}
for any $m\in\N$, where $\{\OO_r\}_{r\in\N}$ is a family of differential operators of order equal or less than $2$, with coefficients explicitly computable in term of local data, and where the differential operator $O(t^{m+1})$ has its coefficients and their derivatives up to order $k$ dominated by $C_k t^{m+1}$ for any $k\in\N$.

The family of polynomials $\{\J{r}\}_{r\in\N}$ defined in \eqref{exp J} can then be computed explicitly by induction using \eqref{expLt2}. In particular, 
the following lemma, which has been established in \cite[\S 2.2]{MM06}, gives the first three elements of this family: 

\begin{lem}\label{J}
For $\OO_1$ and $\OO_2$ defined by \eqref{expLt2}, the following formulas hold:

\begin{equation}
\label{J0}
\J{0}\PP=P,~~i.e.~~\J{0}=\ID,
\end{equation}

\begin{equation}
\label{J1}
\J{1}\PP=-(\LL^0_2)^{-1}P^{\perp}\OO_1P-P\OO_1 (\LL^0_2)^{-1}P^{\perp},
\end{equation}

\begin{equation}
\begin{split}
\label{J2}
\J{2}&\PP=(\LL^0_2)^{-1}P^{\perp}\OO_1 (\LL^0_2)^{-1}P^{\perp}\OO_1 P-(\LL^0_2)^{-1}P^{\perp}\OO_2 P\\
& +P\OO_1 (\LL^0_2)^{-1}P^{\perp}\OO_1 (\LL^0_2)^{-1}P^{\perp}-P\OO_2 (\LL^0_2)^{-1}P^{\perp}\\
+ & (\LL^0_2 )^{-1}P^{\perp}\OO_1 P\OO_1 (\LL^0_2)^{-1}P^{\perp}-P\OO_1(\LL^0_2)^{-2}P^{\perp}\OO_1 P.
\end{split}
\end{equation}

Moreover, $\OO_1$ commutes with any $A\in\End(E_{x_0})$, and we have the formula
\begin{equation}
\label{POP=0}
P\OO_1 P=0.
\end{equation}

In particular, $\J{0}$ and $\J{1}$ commute with any $A\in\End(E_{x_0})$.

\end{lem}

See also \cite[\S 4.1.7]{MM07} for a detailed exposition in the complex case. The assertion that $\OO_1$ commutes with endomorphisms of $E_{x_0}$ is clear from the explicit description given in \cite[\S 2.2]{MM06}. This, together with the fact that $P$ and $\LL^0_2$ act on $\End(E_{x_0})$ by scalar multiplication, implies the last assertion. We point out that due to \eqref{POP=0}, we give in \eqref{J2} a simpler formula than the one appearing in \cite[\S 4.1.7]{MM07}.

Following the notations at the beginning of this section, for any $f\in\cinf(X,\End(E))$, we write $f_{x_0}$ for the induced function on $B(0,\epsilon_0)\subset T_{x_0}X$ with values in $\End(E_{x_0})$ through the trivialization described in \ref{local model}. Note that in particular, $f_{x_0}(0)=f(x_0)$.

To describe the families $\{\QQ{r}(f)\}_{r\in\N}$ and $\{\QQ{r}(f,g)\}_{r\in\N}$ defined in \eqref{exp Q} and \eqref{exp Q2} respectively, we will need the introduction of the kernel calculus presented in \ref{local model}. This is done in the next lemma, which gives as well a formula for the coefficient $C_1(f,g)$ of \eqref{compo}. 

\begin{lem}{(\citep[Lem.4.6,~\S  4.3]{MM08b})}
\label{Q} For $f$ and $g\in\cinf(X,\End(E))$, the following formulas hold:

\begin{equation}
\label{Q1}
\QQ{r}(f,g)=\sum\limits_{r_1+r_2=r} \KK{\QQ{r_1}(f)}{\QQ{r_1}(g)},
\end{equation}

\begin{equation}
\label{Q2}
\QQ{r}(f)=\sum\limits_{r_1+r_2+|\alpha|=r} \KK{J_{r_1,x_0}}{\sum\limits_{|\alpha|=2}\frac{\partial^2 f_{x_0}}{\partial Z^\alpha}(0)\frac{Z^\alpha}{\alpha !} J_{r_2,x_0}},
\end{equation}

\begin{equation}
\label{Q3}
\QQ{1}(f)=f(x_0)\J{1}+\KK{\J{0}}{\sum_{j=1}^{2n}\frac{\partial f_{x_0}}{\partial Z_j}(0)Z_j \J{0}},
\end{equation}

and for $C_1(f,g)\in\cinf(X,\End(E))$ defined by \eqref{compo}, we have

\begin{equation}
\label{C1}
C_1(f,g)(x_0)\ID=\QQ{2}(f,g)(0,0)-\QQ{2}(fg)(0,0).
\end{equation}

\end{lem}

Note that the assertion \eqref{Q1} follows formally from the definitions of $\QQ{r}(f,g)$ and $\QQ{r}(f)$ in \ref{expansion} and the definition of $\KK{.}{.}$ in \eqref{def KK}. In the same way, considering the Taylor expansion of $f_{x_0}$ at $0$, equation \eqref{Q2} follows formally from \ref{offdiag}, \ref{expansion} and the definition of $T_{f,p}$ in \eqref{Berezin-Toeplitz quantization}.
 
As an illustration, let us give the calculation leading to \eqref{Q3}. Notice first that from the identity $T_{f,p}=P_p$ for $f=1$, the definitions of $\J{r}, \QQ{r}$ in \eqref{exp J}, \eqref{exp Q} and \eqref{Q1}, we get
\begin{equation}\label{Q3dem1}
\J{1}=\KK{\J{0}}{\J{1}}+\KK{\J{1}}{\J{0}}.
\end{equation}

But $\J{0}$ and $\J{1}$ commute with any $A\in\End(E_{x_0})$ by \ref{J}. Thus
\begin{equation}\label{Q3dem2}
\begin{split}
\KK{\J{1}}{f(x_0)\J{0}} & =f(x_0)\KK{\J{1}}{\J{0}},\\
 \\
\KK{\J{0}}{f(x_0)\J{1}} & =f(x_0)\KK{\J{0}}{\J{1}}.
\end{split}
\end{equation}

The assertion \eqref{Q3} then follows directly from \eqref{Q3dem1} and \eqref{Q3dem2}.

Finally, \eqref{C1} follows from \eqref{Toeplitzfle1}, and can be found in \citep[(4.82)]{MM08b}. Notice that \eqref{C1} says in particular that the right hand side preserves the degree and vanishes on elements of degree $>0$, a fact which is absolutely not obvious from the formulas \eqref{Q1} and \eqref{Q2}.

Thanks to the known eigenvalues of $(\LL^0_2)^{-1}P^{\perp}$ coming from \ref{spectrumL} and the explicit expression of $\OO_1$ given in \cite[\S 2.2]{MM06}, the terms appearing in \ref{Q} can be computed explicitly. We will only need the following special cases.

\begin{lem}{(\citep[(2.33),\ (2.34)]{MM06})}
\label{L} Let $\<.,.\>$ be the $\C$-bilinear product on $T_{x_0}X\otimes\C$ induced by $g^{TX}$ and let $\nabla^X$ be the connection induced on tensors by the Levi-Civita connection $\nabla^{TX}$. Then the following formulas hold:

\begin{multline}
\label{L1}
(P\OO_1 (\LL^0_2)^{-1}P^{\perp})(0,Z')=\frac{2\sqrt{-1}}{3} \sum_{i,l,m=1}^nz'_i\langle(\nabla^X_{\frac{\partial}{\partial z_i}} J)\del{z_l},\del{z_m}\rangle\ID\\
 i_{\del{\bar{z}_m}}i_{\del{\bar{z}_l}}\PP(0,Z'),
\end{multline}

\begin{multline}
\label{L2}
(P\OO_1 (\LL^0_2)^{-1}P^{\perp})(Z,0)=\frac{\sqrt{-1}}{3\sum_{i,l,m=1}^n}z_i\langle(\nabla^X_{\frac{\partial}{\partial z_i}} J)\del{z_l},\del{z_m}\rangle\ID\\
i_{\del{\bar{z}_l}}i_{\del{\bar{z}_m}}\PP(Z,0),
\end{multline}

\begin{multline}
\label{L3}
((\LL^0_2)^{-1}P^{\perp}\OO_1P)(Z,0)=-\frac{\sqrt{-1}}{6}\sum_{i,l,m=1}^n\bar{z}_i\langle(\nabla^X_{\frac{\partial}{\partial \bar{z}_i}} J)\del{z_l},\del{z_m}\rangle\\
 d\bar{z}_ld\bar{z}_m\ID\PP(Z,0).
\end{multline}

In the last formula, $d\bar{z}_ld\bar{z}_m$ denotes the wedge product in $\Spin$ by $d\bar{z}_ld\bar{z}_m$.

\end{lem}

\subsection{Calculation of the second coefficient}

Let $f$ and $g\in\cinf(X,\End(E))$ be fixed in all the sequel. In this last part, we will use the results summarized in the previous sections in order to compute the coefficient $C_1(f,g)\in\cinf(X,\End(E))$ defined in \eqref{compo}, thus giving a proof of \ref{goal}. 

Recall that by \ref{K}, the polynomials $\J{1}$ and $\J{2}$ commute with any $A\in\End(E_{x_0})$, thus in particular with $h(x_0)=h_{x_0}(0)$ for any $h\in\cinf(X,\End(E))$. Thus from \eqref{K2}, \eqref{K5}, \eqref{J0} and \eqref{Q2}, we have
\begin{equation} 
\label{CC1}
\QQ{0}(h)=\KK{\ID}{h(x_0)\ID}=h(x_0)\ID.
\end{equation}

Let us develop the terms in the expression of $C_1(f,g)$ given by \eqref{C1}. By \eqref{Q1} and \eqref{CC1}, we get on one hand
\begin{equation}
\label{QQ2(f,g)}
\begin{split}
\QQ{2} & (f,g)=\\
& \KK{f(x_0)\J{0}}{\QQ{2}(g)}+\KK{\QQ{2}(f)}{g(x_0)\J{0}}+\KK{Q_{1,x_0}(f)}{Q_{1,x_0}(g)}.
\end{split}
\end{equation}

On another hand, we get from \ref{K}, \ref{J} and \eqref{Q2} that for any $h\in\cinf(X,\End(E))$, thus in particular for $h=fg$,
\begin{equation}
\label{QQ2(h)}
\begin{split}
& \QQ{2}(h)=h(x_0)\KK{\J{0}}{\J{2}}+\KK{\J{2}}{\J{0}}h(x_0)\\
& +h(x_0)\KK{\J{1}}{\J{1}}+\sum\limits_{j=1}^{2n}\frac{\partial h_{x_0}}{\partial Z_j}(0)\KK{\J{1}}{Z_j \J{0}}\\
& +\sum\limits_{j=1}^{2n}\frac{\partial h_{x_0}}{\partial Z_j}(0)\KK{\J{0}}{Z_j \J{1}}+\sum\limits_{|\alpha|=2}\frac{\partial^2 h_{x_0}}{\partial Z^\alpha}(0)\KK{\J{0}}{\frac{Z^\alpha}{\alpha !} \J{0}}.
\end{split}
\end{equation}

The computation of $C_1(f,g)$ will be done in three steps: we will first use \ref{K} to simplify these expressions everywhere it's possible, then use \ref{J} and formal calculus of operators to cancel most of the terms. Finally, we will use \ref{L} and the kernel calculus of the previous section to handle the terms containing $Z$.

We first develop one by one the terms of \eqref{QQ2(f,g)}. Expanding $\QQ{2}(g)$ inside the expression $\KK{f(x_0)\J{0}}{\QQ{2}(g)}$ using \eqref{QQ2(h)} and simplifying with \ref{K},
\begin{equation}
\label{KKfJ0QQ2g}
\begin{split}
& \KK{f(x_0)\J{0}}{\QQ{2}(g)}=f(x_0)g(x_0)\KK{\J{0}}{\J{2}}\\
& +f(x_0)\KKK{\J{0}}{\J{2}}{\J{0}}g(x_0)+f(x_0)g(x_0)\KKK{\J{0}}{\J{1}}{\J{1}}\\
&+f(x_0)\sum\limits_{j=1}^{2n}\frac{\partial g_{x_0}}{\partial Z_j}(0)\KKK{\J{0}}{\J{1}}{Z_j \J{0}}\\
& +f(x_0)\sum\limits_{j=1}^{2n}\frac{\partial g_{x_0}}{\partial Z_j}(0)\KK{\J{0}}{Z_j\J{1}}+
f(x_0)\sum\limits_{|\alpha|=2}\frac{\partial^2 g_{x_0}}{\partial Z^\alpha}(0)\KK{\J{0}}{\frac{Z^\alpha}{\alpha !}\J{0}}.
\end{split}
\end{equation}

We expand in the same way $\QQ{2}(f)$ inside $\KK{\QQ{2}(f)}{g(x_0)\J{0}}$ using \ref{K} and \eqref{QQ2(h)},

\begin{equation}
\label{KKQQ2fgJ0}
\begin{split}
& \KK{\QQ{2}(f)}{g(x_0)\J{0}}=\KK{\J{2}}{\J{0}}f(x_0)g(x_0)\\
& +f(x_0)\KKk{\J{0}}{\J{2}}{\J{0}}g(x_0)+f(x_0)g(x_0)\KKk{\J{1}}{\J{1}}{\J{0}}\\
& +\sum_{j=1}^{2n}\frac{\partial f_{x_0}}{\partial Z_j}(0)g(x_0)\KKk{\J{1}}{Z_j\J{0}}{\J{0}}\\
& +\sum_{j=1}^{2n}\frac{\partial f_{x_0}}{\partial Z_j}(0)g(x_0)\KKk{\J{0}}{Z_j\J{1}}{\J{0}}\\
& +\sum\limits_{|\alpha|=2}\frac{\partial^2 f_{x_0}}{\partial Z^\alpha}(0)g(x_0)\KKk{\J{0}}{\frac{Z^\alpha}{\alpha !}\J{0}}{\J{0}}.
\end{split}
\end{equation}

We then use \ref{K} and \eqref{Q3} to expand $\QQ{1}(f)$ and $\QQ{1}(g)$ inside the last term of \eqref{QQ2(f,g)},
\begin{equation}
\label{KKQQ1fQQ1g}
\begin{split}
& \KK{\QQ{1}(f)}{\QQ{1}(g)}=f(x_0)g(x_0)\KK{\J{1}}{\J{1}}\\
& +f(x_0)\sum\limits_{j=1}^{2n}\frac{\partial g_{x_0}}{\partial Z_j}(0)\KKK{\J{1}}{\J{0}}{Z_j\J{0}}\\
& +\sum_{j=1}^{2n}\frac{\partial f_{x_0}}{\partial Z_j}(0)g(x_0)\KKk{\J{0}}{Z_j\J{0}}{\J{1}}\\
& +\sum\limits_{i,j=1}^{2n}\KKKK{\J{0}}{\frac{\partial f_{x_0}}{\partial Z_i}(0)Z_i\J{0}}{\J{0}}{\frac{\partial g_{x_0}}{\partial Z_j}(0)Z_j\J{0}}.
\end{split}
\end{equation}

Let us now add together \eqref{KKfJ0QQ2g}, \eqref{KKQQ2fgJ0} and \eqref{KKQQ1fQQ1g} to get \eqref{QQ2(f,g)}. We first point out some useful cancellations: by \eqref{def KK} and \eqref{J1}, we get
\begin{equation}
\label{PO1L2O1P}
\begin{split}
\KKK{\J{0}}{\J{1}}{\J{1}}\PP & =P(\J{1}\PP)(\J{1}\PP)P\\
& =P\OO_1 (\LL^0_2)^{-2}P^{\perp}\OO_1 P,
\end{split}
\end{equation}

\begin{equation}
\KKk{\J{1}}{\J{1}}{\J{0}}\PP=P\OO_1 (\LL^0_2)^{-2}P^{\perp}\OO_1 P.
\end{equation}

On another hand, by \eqref{def KK}, \eqref{J0} and \eqref{J2}, we get
\begin{equation}
\label{-PO1L2O1P}
\begin{split}
\KKK{\J{0}}{\J{2}}{\J{0}}\PP & =\KKk{\J{0}}{\J{2}}{\J{0}}\PP\\
& =P(\J{2}\PP)P\\
& =-P\OO_1 (\LL^0_2)^{-2}P^{\perp}\OO_1 P.
\end{split}
\end{equation}

As the operator on the right hand side of the equations \eqref{PO1L2O1P}-\eqref{-PO1L2O1P} commutes with constant endomorphisms by \ref{J}, equations \eqref{PO1L2O1P}-\eqref{-PO1L2O1P} and \eqref{K5} show that the second and third terms of \eqref{KKfJ0QQ2g} cancel each other, as well as the second and third terms of \eqref{KKQQ2fgJ0}.

Now, in the difference $\QQ{2}(f,g)-\QQ{2}(fg)$, the first three terms of \eqref{QQ2(h)} with $h=fg$ cancel with the first terms of \eqref{KKfJ0QQ2g}, \eqref{KKQQ2fgJ0} and \eqref{KKQQ1fQQ1g} respectively.

Using \eqref{J0} and the cancellations above, we are now ready to describe the terms of $\QQ{2}(f,g)-\QQ{2}(fg)$. Let us define $I_1, I_2, I_3$ and $I_4$, polynomials in $Z,Z'\in\R^{2n}$ with values in $\End(\Spin\otimes E)_{x_0}$, by the following formulas:

\begin{equation}\label{I1}
\begin{split}
& I_1=-\sum\limits_{|\alpha|=2}\frac{\partial^2 (fg)_{x_0}}{\partial Z^\alpha}(0)\KK{\ID}{\frac{Z^\alpha}{\alpha !} \ID}\\
& +f(x_0)\sum\limits_{|\alpha|=2}\frac{\partial^2 g_{x_0}}{\partial Z^\alpha}(0)\KK{\ID}{\frac{Z^\alpha}{\alpha !}\ID}\\
& +\sum\limits_{|\alpha|=2}\frac{\partial^2 f_{x_0}}{\partial Z^\alpha}(0)g(x_0)\KKk{\ID}{\frac{Z^\alpha}{\alpha !}\ID}{\ID}\\
 +\sum\limits_{i,j=1}^{2n} & \KKKK{\ID}{\frac{\partial f_{x_0}}{\partial Z_i}(0)Z_i\ID}{\ID}{\frac{\partial g_{x_0}}{\partial Z_j}(0)Z_j\ID},
\end{split}
\end{equation}

\begin{equation}\label{I2}
\begin{split}
I_2 &=f(x_0)\sum\limits_{j=1}^{2n}\frac{\partial g_{x_0}}{\partial Z_j}(0)\KKK{\ID}{\J{1}}{Z_j \ID}\\
& +\sum_{j=1}^{2n}\frac{\partial f_{x_0}}{\partial Z_j}(0)g(x_0)\KKk{\ID}{Z_j\J{1}}{\ID},
\end{split}
\end{equation}

\begin{equation}\label{I3}
\begin{split}
I_3 &=-\sum\limits_{j=1}^{2n}\frac{\partial (fg)_{x_0}}{\partial Z_j}(0)\KK{\ID}{Z_j \J{1}}\\
& +f(x_0)\sum\limits_{j=1}^{2n}\frac{\partial g_{x_0}}{\partial Z_j}(0)\KK{\ID}{Z_j\J{1}}\\
& +\sum_{j=1}^{2n}\frac{\partial f_{x_0}}{\partial Z_j}(0)g(x_0)\KKk{\ID}{Z_j\ID}{\J{1}},
\end{split}
\end{equation}

\begin{equation}\label{I4}
\begin{split}
I_4 &=-\sum\limits_{j=1}^{2n}\frac{\partial (fg)_{x_0}}{\partial Z_j}(0)\KK{\J{1}}{Z_j \ID}\\
& +\sum_{j=1}^{2n}\frac{\partial f_{x_0}}{\partial Z_j}(0)g(x_0)\KKk{\J{1}}{Z_j\ID}{\ID}\\
& +f(x_0)\sum\limits_{j=1}^{2n}\frac{\partial g_{x_0}}{\partial Z_j}(0)\KKK{\J{1}}{\ID}{Z_j\ID}.
\end{split}
\end{equation}

Then using \eqref{J0} and by \eqref{QQ2(f,g)}-\eqref{I4}, we get
\begin{equation}
\label{QQ2(f,g)-QQ2(fg)}
\QQ{2}(f,g)-\QQ{2}(fg)=I_1+I_2+I_3+I_4.
\end{equation}

Recall that by definition, the terms $I_1, I_2, I_3$ and $I_4$ in \eqref{QQ2(f,g)-QQ2(fg)} are polynomials in $Z, Z'\in\R^{2n}$. Thus by \eqref{C1}, in order to compute $c_1(f,g)$, it suffices to compute the values of $I_1, I_2, I_3$ and $I_4$ at $Z=Z'=0$. We compute those values one by one in the following propositions, using the kernel calculus described in \ref{local model}.


\begin{prop}\label{fgh}
For all $f, g$ and $h\in\cinf(\R^{2n},\End(E_{x_0}))$, the following formulas hold:
\begin{equation}\label{Zalpha}
\begin{split}
\sum\limits_{|\alpha|=2}\frac{\partial^2 h}{\partial Z^\alpha}(0) & \KKk{\ID}{\frac{Z^\alpha}{\alpha !}\ID}{\ID}(0,0)\\
& =\sum\limits_{|\alpha|=2}\frac{\partial^2 h}{\partial Z^\alpha}(0)\KK{\ID}{\frac{Z^\alpha}{\alpha !} \ID}(0,0)=\frac{1}{\pi}\sum\limits_{i=1}^n\frac{\partial^2 h}{\partial z_i\partial\bar{z}_i}(0)\ID,
\end{split}
\end{equation}
\begin{equation}
\label{zibarzi}
\begin{split}
\sum_{i,j=1}^{2n}\frac{\partial f}{\partial Z_j}(0)\frac{\partial g}{\partial Z_i}(0)\KKKK{\ID}{Z_i\ID}{\ID}{ & Z_j\ID}(0,0)\\
 & =\frac{1}{\pi}\sum_{i=1}^n\frac{\partial f}{\partial \bar{z}_i}(0)\frac{\partial g}{\partial z_i}(0)\ID,
\end{split}
\end{equation}
so that the value at $Z=Z'=0$ of $I_1$ in \eqref{QQ2(f,g)-QQ2(fg)} is given by

\begin{equation}\label{I100}
I_1(0,0)=-\frac{1}{\pi}\sum_{j=1}^{n}\frac{\partial f_{x_0}}{\partial z_j}(0)\frac{\partial g_{x_0}}{\partial \bar{z}_j}(0)\ID.
\end{equation}

\end{prop}

\begin{proof}
From \eqref{K1}, \eqref{K2} and \eqref{K4}, we get
\begin{equation}
\label{alpha2}
\begin{split}
\sum\limits_{|\alpha|=2}\KKk{\ID}{Z^\alpha \ID}{\ID} & =\sum\limits_{|\alpha|=2}\KKk{(Z')^\alpha\ID}{\ID}{\ID}\\
& =\sum\limits_{|\alpha|=2}\KK{(Z')^\alpha\ID}{\ID}\\
& =\sum\limits_{|\alpha|=2}\KK{\ID}{Z^\alpha\ID},
\end{split}
\end{equation}
which shows the first equality of \eqref{Zalpha}. On another hand, by \ref{K}, \eqref{Z=z+zbar} and \eqref{kernel calculus}, we compute for $h\in\cinf(\R^{2n},\End(E_{x_0}))$,
\begin{equation}
\label{halpha2}
\begin{split}
\sum\limits_{|\alpha|=2} & \frac{\partial^2 h}{\partial Z^\alpha}(0)\KK{\ID}{\frac{Z^\alpha}{\alpha !} \ID} = \frac{1}{\pi}\sum\limits_{i=1}^n\frac{\partial^2 h}{\partial z_i\partial \bar{z}_i}(0)\ID\\
 & +\sum\limits_{i,j=1}^n\left(\frac{\partial^2 h}{\partial z_i\partial z_j}(0)z_iz_j+\frac{\partial^2 h}{\partial \bar{z}_i\partial \bar{z}_j}(0)\bar{z}_i'\bar{z}_j'+\frac{\partial^2 h}{\partial z_i\partial \bar{z}_j}(0)z_i\bar{z}_j'\right)\ID.
\end{split}
\end{equation}

Evaluating \eqref{halpha2} at $Z=Z'=0$ then gives \eqref{Zalpha}. 

By \eqref{Z=z+zbar} and \eqref{kernel calculus}, we get for any $f\in\cinf(X,\End(E))$,
\begin{equation}
\label{dfdzi}
\begin{split}
\KK{\ID}{\sum\limits_{i=1}^{2n}\frac{\partial f_{x_0}}{\partial Z_i}(0)Z_i\ID}=\sum\limits_{i=1}^n\left(\frac{\partial f_{x_0}}{\partial z_i}(0)z_i+\frac{\partial f_{x_0}}{\partial \bar{z}_i}(0)\bar{z}_i'\right)\ID.
\end{split}
\end{equation}

By \eqref{K4}, \eqref{K5}, \eqref{kernel calculus} and \eqref{dfdzi}, in the same way  than in \eqref{halpha2}, we get at $Z=Z'=0$ for all $f,g\in\cinf(X,\End(E))$,
\begin{equation}
\label{Kdfdzidgdzj}
\begin{split}
\sum\limits_{i,j=1}^{2n}\KKKK{\ID}{\frac{\partial f}{\partial Z_i}(0) & Z_i\ID}{\ID}{\frac{\partial g(0)}{\partial Z_j}Z_j\ID}(0,0)\\
& =\sum_{i=1}^n\frac{\partial f}{\partial \bar{z}_i}(0)\frac{\partial g}{\partial z_i}(0)\KK{\bar{z}_i'\ID}{z_i\ID}(0,0)\\
& =\frac{1}{\pi}\sum_{i=1}^n \frac{\partial f}{\partial \bar{z}_i}(0)\frac{\partial g}{\partial z_i}(0)\ID.
\end{split}
\end{equation}

As $\ID$ commutes with any $A\in\End(E_{x_0})$, from \eqref{Kdfdzidgdzj} we get \eqref{zibarzi}.

Finally, by the formula \eqref{I1} for $I_1$, equations \eqref{Zalpha} and \eqref{zibarzi} give 
\begin{equation}\label{I100=L}
\begin{split}
& I_1(0,0)=\\
& \frac{1}{\pi}\sum\limits_{i=1}^{n}\left(-\frac{\partial^2 (fg)_{x_0}}{\partial z_i\partial \bar{z}_i}(0)+f(x_0)\frac{\partial g_{x_0}}{\partial z_i\partial \bar{z}_i}(0)+\frac{\partial^2 f_{x_0}}{\partial z_i\partial \bar{z}_i}(0)g(x_0)+\frac{\partial f_{x_0}}{\partial \bar{z}_i}(0)\frac{\partial g_{x_0}}{\partial z_i}(0)\right)\ID.
\end{split}
\end{equation}

The equality \eqref{I100} then follows immediately from \eqref{I100=L} by Leibniz rule.
\end{proof}

Let us point out that all the terms of $I_2, I_3$ and $I_4$ in \eqref{I1}, \eqref{I2} and \eqref{I3} contain $\J{1}$. We already see from its expression in \eqref{J1} that the computations will involve the explicit expression of $\OO_1$, and we will thus need to use \ref{L}.

\begin{prop}\label{KJ0=KJ1=0}
For any $1\leq i\leq 2n$, the following formulas hold:
\begin{equation}
\label{KJ0J1ZiJ0}
\begin{split}
\KKK{\ID}{\J{1}}{Z_i\ID}(0,0)=0,\\
\KKK{\ID}{Z_i\J{1}}{\ID}(0,0)=0.
\end{split}
\end{equation}
so that 
\begin{equation}\label{I200}
I_2(0,0)=0,
\end{equation}
i.e. the polynomial $I_2$ in \eqref{I2} vanishes at $Z=Z'=0$.

\end{prop}

\begin{proof} First, we've got by \eqref{def KK} and \eqref{J1},
\begin{equation}
\label{KJ0J1ZiJ0PP}
\begin{split}
\KKK{\ID}{\J{1}}{Z_i\ID}\PP & =P(\J{1}\PP)(Z_i\ID\PP)\\
& =-P\OO_1 (\LL^0_2)^{-1}P^{\perp}(Z_i\ID\PP).
\end{split}
\end{equation}

With the convention that operators always act on the $Z$ variable, equation \eqref{KJ0J1ZiJ0PP} gives us
\begin{equation}
\label{KJ0J1ZiJ0PPZZ}
\begin{split}
\KKK{\ID}{\J{1} & }{Z_i\ID}(Z,Z')\PP(Z,Z')\\
& =-P\OO_1 (\LL^0_2)^{-1}P^{\perp}(Z_i\ID\PP(Z,Z'))\\
& =-\int_{\R^{2n}} (P\OO_1 (\LL^0_2)^{-1}P^{\perp})(Z,Z'')\ID Z_i''\PP(Z'',Z')dZ''.
\end{split}
\end{equation}

Recall that by \eqref{PP=1}, $\PP(Z,Z')$ commutes with $\ID$. By the definition of $\ID$ in \eqref{ID}, we have $i_{\del{\bar{z}_m}}i_{\del{\bar{z}_l}}\ID=0$, so that \eqref{L1} implies
\begin{equation}\label{PO1L02=0}
(P\OO_1 (\LL^0_2)^{-1}P^{\perp})(0,Z'')\ID=0.
\end{equation}

We thus deduce from \eqref{KJ0J1ZiJ0PPZZ} and \eqref{PO1L02=0} that
\begin{equation}
\label{KJ0J1ZiJ0PP00}
\begin{split}
\KKK{\ID}{\J{1} & }{Z_i\ID}(0,0)\\
& =-\int_{\R^{2n}} (P\OO_1 (\LL^0_2)^{-1}P^{\perp})(0,Z'')\ID Z''_i\PP(Z'',0)dZ''=0.
\end{split}
\end{equation}

Equation \eqref{KJ0J1ZiJ0PP00} is precisely the first equality of \eqref{KJ0J1ZiJ0}.

On another hand, by \eqref{K1} and \eqref{K3},
\begin{equation}
\label{KJ0ZiJ1J0}
\begin{split}
\KKK{\ID}{Z_i\J{1}}{\ID} & =\KK{\ID}{Z_i\KK{\J{1}}{\ID}}\\
 & =\KKK{Z_i'\ID}{\J{1}}{\ID}.
\end{split}
\end{equation}

By \eqref{J1},
\begin{equation}
\label{KJ1J0}
\begin{split}
\KK{\J{1}}{\ID}\PP & =\left(-(\LL^0_2)^{-1}P^{\perp}\OO_1P-P\OO_1 (\LL^0_2)^{-1}P^{\perp}\right)P\\
& =-(\LL^0_2)^{-1}P^{\perp}\OO_1P.
\end{split}
\end{equation}

But $\ID d\bar{z}_ld\bar{z}_m=0$ by \eqref{ID}, so analogous to \eqref{PO1L02=0} and by \eqref{L3}, we get
\begin{equation}\label{L02PO1=0}
\ID((\LL^0_2)^{-1}P^{\perp}\OO_1P)(Z'',0)=0,
\end{equation}
so that by \eqref{KJ1J0} and \eqref{L02PO1=0},
\begin{equation}
\label{KZi'J0J1J0}
\begin{split}
\KKK{Z_i'\ID}{\J{1} & }{\ID}(0,0)=-(Z_i'P(\LL^0_2)^{-1}P^{\perp}\OO_1P)(0,0)\\
& =-\int_{\R^{2n}} Z''_i\PP(0,Z'')\ID((\LL^0_2)^{-1}P^{\perp}\OO_1P)(Z'',0) dZ''=0.
\end{split}
\end{equation}

From \eqref{KJ0ZiJ1J0} and \eqref{KZi'J0J1J0}, we get the second equality of \eqref{KJ0J1ZiJ0}.

Finally, equation \eqref{I200} follows immediately from \eqref{KJ0J1ZiJ0} and the definition of $I_2$ in \eqref{I2}.
\end{proof}

\begin{prop}\label{KJ0+KJ1} 
For any $1\leq i\leq 2n$, the following formulas hold:
\begin{equation}
\label{KJ0ZiJ1}
\begin{split}
\KK{\ID}{Z_i\J{1}}(0,0 ) & =\KKK{\ID}{\ID}{Z_i\J{1}}(0,0)\\
& =\KKk{\ID}{Z_i\ID}{\J{1}}(0,0).
\end{split}
\end{equation}
\begin{equation}
\label{KJ1ZiJ0}
\begin{split}
\KK{\J{1}}{Z_i\ID}(0,0) & =\KKk{\J{1}}{Z_i\ID}{\ID}(0,0)\\
& =\KKK{\J{1}}{\ID}{Z_i\ID}(0,0).
\end{split}
\end{equation}
so that the values at $Z=Z'=0$ of $I_3$ and $I_4$ in \eqref{I3} and \eqref{I4} are respectively
\begin{equation}\label{I300I400}
I_3(0,0)=I_4(0,0)=0,
\end{equation}
i.e. the polynomials $I_3$ and $I_4$ in \eqref{I3} and \eqref{I4} vanish at $Z=Z'=0$.

\end{prop}

\begin{proof}

From \eqref{K1}, \eqref{K2} and \eqref{K3}, we immediately get the first lines of \eqref{KJ0ZiJ1} and \eqref{KJ1ZiJ0}. Next, we show that 
\begin{equation}\label{KJ0ZiJ0J1=KJ0ziJ1=0}
\KKk{\ID}{Z_i\ID}{\J{1}}(0,0)=\KK{\ID}{Z_i\J{1}}(0,0).
\end{equation}

By \eqref{kernel calculus}, remembering that $Z=(z+\bar{z})/2$, we get on one hand
\begin{equation}
\label{Kzizi'J0J1}
\KKk{\ID}{Z_i\ID}{\J{1}}=\frac{1}{2}\KK{(z_i+\bar{z}_i')\ID}{\J{1}}.
\end{equation}

By \eqref{K3} and \eqref{kernel calculus}, we get on another hand
\begin{equation}
\label{Kzi'zi'J0J1}
\KK{\ID}{Z_i\J{1}}=\KK{Z_i'\ID}{\J{1}}=\frac{1}{2}\KK{(z_i'+\bar{z}_i')\ID}
{\J{1}}.
\end{equation}

By \eqref{Kzizi'J0J1} and \eqref{Kzi'zi'J0J1}, to get \eqref{KJ0ZiJ0J1=KJ0ziJ1=0} it suffices to prove the equality
\begin{equation}
\label{KziJ0J100=Kzi'J0J100}
\KK{z_i\ID}{\J{1}}(0,0)=\KK{z_i'\ID}{\J{1}}(0,0)=0.
\end{equation}

At first, by \eqref{K4} we have
\begin{equation}
\label{KziJ0J100}
\KK{z_i\ID}{\J{1}}=z_i\KK{\ID}{\J{1}},
\end{equation}
so that
\begin{equation}\label{KziJ0J100=0}
\KK{z_i\ID}{\J{1}}(0,0)=0.
\end{equation}

Then, by \eqref{kernel calculus} and \eqref{J1},
\begin{equation}
\label{Kzi'J0J1}
\begin{split}
\KK{z_i'\ID}{\J{1}}\PP & =\KK{\ID}{z_i\J{1}}\PP\\
& =-P\left(z_i(\LL^0_2)^{-1}P^{\perp}\OO_1P\right)-P\left(z_iP\OO_1(\LL^0_2)^{-1}P^{\perp}\right).
\end{split}
\end{equation}

As $\ID d\bar{z}_ld\bar{z}_m=0$ by \eqref{ID}, analogous to \eqref{PO1L02=0} and \eqref{L02PO1=0}, by \eqref{L3}, 
\begin{equation}
\ID((\LL^0_2)^{-1}P^{\perp}\OO_1P)(Z,0)=0.
\end{equation}

As $P=\PP\ID$ by \eqref{def P}, we deduce that 
\begin{equation}
\begin{split}\label{PziLPO1P00=0}
P(z_i(\LL^0_2)^{-1}P^{\perp}\OO_1P)(0,0) & =
\int_{\R^{2n}}\PP(0,Z'')z_i''\ID((\LL^0_2)^{-1}P^{\perp}\OO_1P)(Z'',0)dZ''\\
& =0,
\end{split}
\end{equation}
i.e. the kernel of the first term of the last line of \eqref{Kzi'J0J1} cancels at $Z=Z'=0$. On the other hand, by \eqref{L2} we can write
\begin{equation}\label{HP}
(z_iP\OO_1(\LL^0_2)^{-1}P^{\perp})(Z,0)=H(z)\PP(Z,0),
\end{equation}
with $H(z)\in\End(\Spin\otimes E)_{x_0}$ polynomial in $z\in\C^n$. Recall from \eqref{zbeta} that in this case, again with the convention that operators act on the $Z$ variable, we get
\begin{equation}\label{PHP=HP}
\PP(H(z)\PP)(Z,0)=H(z)\PP(Z,0),
\end{equation}
so that by \eqref{def P}, \eqref{K3K4pre}, \eqref{HP} and \eqref{PHP=HP},
\begin{equation}\label{zivanat0}
\begin{split}
P(z_iP\OO_1(\LL^0_2)^{-1}P^{\perp})(Z,0) & =(z_iP\OO_1(\LL^0_2)^{-1}P^{\perp})(Z,0)\\
& =z_i(P\OO_1(\LL^0_2)^{-1}P^{\perp})(Z,0),
\end{split}
\end{equation}
which vanishes at $Z=0$. By \eqref{Kzi'J0J1}, \eqref{PziLPO1P00=0} and \eqref{zivanat0}, we thus get
\begin{equation}
\label{Kzi'J0J1=0}
\KK{z_i'\ID}{\J{1}}(0,0)=0.
\end{equation}

Equation \eqref{Kzi'J0J1=0}, together with \eqref{Kzizi'J0J1}, \eqref{Kzi'zi'J0J1} and \eqref{KziJ0J100=0}, proves \eqref{KJ0ZiJ1}.

Now concerning \eqref{KJ1ZiJ0}, we are left to show that
\begin{equation}\label{cqfdderniertruc}
\KK{\J{1}}{Z_i\ID}(0,0)=\KKK{\J{1}}{\ID}{Z_i\ID}(0,0).
\end{equation}

By \eqref{kernel calculus}  we have
\begin{equation}
\label{KJ1J0ZiJ0}
\KKK{\J{1}}{\ID}{Z_i\ID}=\frac{1}{2}\KK{\J{1}}{(z_i+\bar{z}_i')\ID}.
\end{equation}

To get \eqref{cqfdderniertruc}, it suffices thus to show that
\begin{equation}
\label{KJ1ziJ000}
\KK{\J{1}}{\bar{z}_i\ID}(0,0)=\KK{\J{1}}{\bar{z}_i'\ID}(0,0)=0.
\end{equation}

The equality on the right of \eqref{KJ1ziJ000} comes from \eqref{K4}. On another hand, by \eqref{J0} and \eqref{J1},
\begin{equation}
\label{KJ1ziJ0}
\KK{\J{1}}{\bar{z}_i\ID}\PP=-P\OO_1(\LL^0_2)^{-1}P^{\perp}(\bar{z}_i\ID\PP)-(\LL^0_2)^{-1}P^{\perp}\OO_1P(\bar{z}_i\ID\PP).
\end{equation}

Now by \eqref{PO1L02=0}, once again the kernel of the first term of the left member of \eqref{KJ1ziJ0} vanishes at $Z=Z'=0$. On another hand, by \eqref{Pzi=Pzi'bar},
\begin{equation}
\label{PziJ0P=Pzi'JOP}
P(\bar{z}_i\ID\PP)=P(\bar{z}_i'\ID\PP).
\end{equation}

We can thus replace $\bar{z}_i$ by $\bar{z}_i'$ in the second term of the left member of \eqref{KJ1ziJ0}, and by \eqref{K3K4pre} we get finally
\begin{equation}
\label{kernel PO1LziP}
((\LL^0_2)^{-1}P^{\perp}\OO_1P\bar{z}_i\ID\PP)(Z,Z')=((\LL^0_2)^{-1}P^{\perp}\OO_1P\ID\PP)(Z,Z')\bar{z}_i',
\end{equation}
which is $0$ at $Z=Z'=0$. Thus the kernel of the second term of \eqref{KJ1ziJ0} cancels as well in this case, which means
\begin{equation}
\label{KJ1ziJ0=0}
\KK{\J{1}}{\bar{z}_i\ID}(0,0)=0,
\end{equation}

Finally, \eqref{KJ1ziJ0=0} implies \eqref{KJ1ziJ000}, which together with \eqref{KJ0ZiJ0J1=KJ0ziJ1=0} concludes the proof of \ref{KJ1ZiJ0}.

The definitions of $I_3$ and $I_4$ in \eqref{I3} and \eqref{I3} and equation \eqref{KJ0ZiJ1} and \eqref{KJ1ZiJ0} respectively give 
\begin{equation}
\begin{split}
I_3(0,0) & =\sum\limits_{i=1}^{2n}\left(f(x_0)\frac{\partial g_{x_0}}{\partial Z_i}(0)+\frac{\partial f_{x_0}}{\partial Z_i}(0)g(x_0)-\frac{\partial (fg)_{x_0}}{\partial Z_i}(0)\right)\KK{\ID}{Z_i\J{1}}(0,0),\\
I_4(0,0) &=\sum\limits_{i=1}^{2n}\left(f(x_0)\frac{\partial g_{x_0}}{\partial Z_i}(0)+\frac{\partial f_{x_0}}{\partial Z_i}(0)g(x_0)-\frac{\partial (fg)_{x_0}}{\partial Z_i}(0)\right)\KK{\J{1}}{Z_i\ID}(0,0),
\end{split}
\end{equation}
and those two formulas vanish by Leibniz rule. We thus get \eqref{I300I400}.

\end{proof}


Using \ref{fgh}, \ref{KJ0=KJ1=0} and \ref{KJ0+KJ1},the kernel of \eqref{QQ2(f,g)-QQ2(fg)} at $Z=Z'=0$ simply is
\begin{equation}
\label{C1fg}
\begin{split}
\QQ{2}(f,g)(0,0)-\QQ{2}(fg)(0,0)=-\frac{1}{\pi}\sum_{j=1}^{n}\frac{\partial f_{x_0}}{\partial z_j}(0)\frac{\partial g_{x_0}}{\partial \bar{z}_j}(0)\ID.
\end{split}
\end{equation}

We thus see that $\QQ{2}(f,g)(0,0)-\QQ{2}(fg)(0,0)$ is in fact of the form $C_1(f,g)(x_0)\ID$ with $C_1(f,g)(x_0)\in\End(E_{x_0})$ given by
\begin{equation}
\label{half poisson}
C_1(f,g)(x_0)=-\frac{1}{\pi}\sum_{j=1}^{n}\frac{\partial f}{\partial z_j}(x_0)\frac{\partial g}{\partial \bar{z}_j}(x_0).
\end{equation}

From \eqref{base} and the definition of the pairing $\<.,.\>$ used in \eqref{fleC_1(f,g)}, we can take this equality to the manifold through our trivialization and we finally get
\begin{equation}
\label{final formula}
C_1(f,g)=-\frac{1}{2\pi}\langle\nabla^{1,0}f,\nabla^{0,1}g\rangle.
\end{equation}

This proves \ref{goal}.


\begin{thebibliography}{99}
  
  \bibitem{BS75}
L.~{B}outet~de~{M}onvel and J.~Sj{\"o}strand, \emph{Sur la singularit{\'e} des
  noyaux de {B}ergman et de {S}zeg{\"o}}, Journ{\'e}es: {\'E}quations aux
  D{\'e}riv{\'e}es Partielles de Rennes (1975), Soc. Math. France, Paris, 1976,
  pp.~123--164. Ast{\'e}risque, no. 34--35.
  
  \bibitem{BG81}
L.~{B}outet~de~{M}onvel and V.~Guillemin, \emph{{The spectral theory of
  Toeplitz operators}}, Annals of Math. Studies, no.~99, Princeton Univ. Press,
  Princeton, NJ, 1981.

\bibitem{BMS94}
M.~Bordemann, E.~Meinrenken, and M.~Schlichenmaier, \emph{{Toeplitz
  quantization of K{\"a}hler manifolds and $gl(N)$, $N\longrightarrow\infty$
  limits}}, Comm. Math. Phys. \textbf{165} (1994), no.~2, 281--296.
  
  \bibitem{Cha03}
  L.~Charles, \emph{Berezin-Toeplitz operators, a semi-classical approach}, Comm. Math. Phys. \textbf{239} (2003), no.~1-2, 1--28.
  
  \bibitem{Cha16a}
  L,~Charles, \emph{Quantization of compact symplectic manifolds}, J. Geom. Anal. \textbf{26} (2016), no.~4, 2664--2710.
  
    \bibitem{CL}
  L,~Charles, \emph{Subprincipal symbol for Toeplitz operators}, Letters in Mathematical Physics \textbf{106} (2016), no.~12, 1673--1694.
  
  \bibitem{DLM06}
X.~Dai, K.~Liu, and X.~Ma, \emph{On the asymptotic expansion of {B}ergman
  kernel}, J. Differential Geom., \textbf{72} (2006), no.~1, 1--41; announced in
  C. R. Math. Acad. Sci. Paris \textbf{339} (2004), no.~3, 193--198.
  
  \bibitem{DV15}
  A.~Della Vedova, \emph{A note on Berezin-Toeplitz quantization of the Laplace operator}, Complex Manifolds \textbf{2} (2015), 131--139.
  
  \bibitem{Fin12}
  J.~Fine, \emph{Quantization and the Hessian of Mabuchi energy}, Cuke Math. J;, \textbf{61} (2012), no.~14, 2753--2798.
  
  \bibitem{Hsi12}
  C.-Y.~Hsiao, \emph{On the coefficients of the asymptotic expansion of the kernele of Berezin-Toeplitz quantization}, Ann. Global Anal. Geom. \textbf{42} (2012), no.~2, 207--245.
  
  \bibitem{KMS16}
  J.~Keller, J.~Meyer, R.~Seyyedali, \emph{Quantization of the Laplace operator on vector bundles,~I}, Math. Ann. \textbf{366} (2016), no.~3-4, 865--907.
  
   \bibitem{MM07}
 X. Ma and G. Marinescu,
 \emph{Holomorphic Morse Inequalities and Bergman Kernels},
 Progress in Mathematics, vol. 254, Birkh\"auser Boston, Inc., Boston, MA, 2007.

 \bibitem{MM06}
  X. Ma and G. Marinescu, \emph{The first coefficients of the asymptotic expansion of the Bergman kernel of the Spin$^c$ Dirac operator},
  Internat. J. Math. \textbf{17} (2006), no.\ 6, 737--759.

  \bibitem{MM08b}
  X. Ma and G. Marinescu, \emph{Toeplitz operators on symplectic manifolds},
  J. Geom. Anal. \textbf{18} (2008), 565--611.
  
   \bibitem{MM12}
  X. Ma and G. Marinescu, \emph{Berezin-Toeplitz quantization on K\"ahler manifolds},
 J. Reine Angew. Math. \textbf{662} (2002), 1--56.

\bibitem{Sch00}
M.~Schlichenmaier, \emph{Deformation quantization of compact {K}{\"a}hler
  manifolds by {B}erezin-{T}oeplitz quantization}, Conf{\'e}rence Mosh{\'e} Flato
  1999, Vol. II (Dijon), Math. Phys. Stud., vol.~22, Kluwer Acad. Publ.,
  Dordrecht, 2000, pp.~289--306.
  
  \bibitem{Xu13}
  H.~Xu, \emph{On a graph theoretic formula of Gammelgaard for Berezin-Toeplitz quantization}, Lett. Math. Phys. \textbf{103} (2013), no.~2, 145--169. 
  
\end{thebibliography}

Université Paris Diderot - Paris 7, Institut de Mathématiques de Jussieu-Paris Rive Gauche, Case 7012,
75205 Paris Cedex 13, France\\
\emph{E-mail adress}: louis.ioos@imj-prg.fr

\end{document}